\documentclass[12pt]{article}
\usepackage[utf8]{inputenc}

\usepackage{amssymb,    enumitem,    xcolor}
\usepackage{tikz}
\usepackage[normalem]{ulem}

\usepackage{graphics,    graphicx,    amsmath}
\usepackage{lmodern,    microtype}

\usepackage{mathtools,    amsthm,    icomma,    upgreek,    xfrac,    amsmath,    dsfont,    float}
\mathtoolsset{showonlyrefs,    mathic}

\usepackage[colorlinks=true,     urlcolor=blue,     linkcolor=blue,     citecolor=blue]{hyperref}

\allowdisplaybreaks

\newcommand{\E}{\mathbf{E}}

\newcommand{\N}{\mathbf{N}}

\theoremstyle{plain}
\newtheorem{theorem}{Theorem}[section]

\newtheorem{proposition}[theorem]{Proposition}
\newtheorem{corollary}[theorem]{Corollary}
\theoremstyle{definition}

\newtheorem{example}{Example}[section]
\newtheorem{remark}[theorem]{Remark}
\newtheorem{fact}{Fact}[section]

\numberwithin{equation}{section}

\def\({\left(}
\def\){\right)}
\def\[{\left[}
\def\]{\right]}

\def\ind{{\mathbb I}}

\def\({\left(}
\def\){\right)}

\def\[{\left[}
\def\]{\right]}

\newcommand{\pr}{\mathbb{P}}
\newcommand{\cP}{{\mathcal P}}

\newcommand{\V}{\mathcal{V}}

\newcommand{\A}{\mathcal{A}}

\newcommand{\covers}{{\mathcal C}}
\newcommand{\cover}{{\mathbb C}}

\newcommand{\kli}{\mathbf{K}}

\newcommand{\Pra}[1]{\pr\left(#1\right)}
\newcommand{\Gnmp}{{\mathcal G}(n,     m,     p)}

\newcommand{\Gnp}{{\mathcal G}(n,   {\bf p})}

\newcommand{\pp}{{\bf p}}
\newcommand{\calP}{{\cal P}}
\newcommand{\Ind}{{\mathbb I}}
\newcommand{\Jind}{{\mathbb J}}
\newcommand{\Lind}{{\mathbb L}}

\allowdisplaybreaks

\date{\vspace{-5ex}}

\title{
\huge
Subgraph probabilities\\ in random intersection graphs.
} 

\author{ 
Katarzyna Rybarczyk\thanks{Faculty of Mathematics and Computer Science,       Adam Mickiewicz University,       Pozna\'n,       Poland.
e-mail: {\tt kryba@amu.edu.pl}. ORCID: 0000-0002-3998-6482}
\and
Grzegorz Serafin\thanks{Faculty of Pure and Applied Mathematics,        Wroc{\l}aw University of Science and Technology,       Ul. Wybrze\.ze Wyspia\'nskiego 27,        Wroc{\l}aw,       Poland.
e-mail: {\tt grzegorz.serafin@pwr.edu.pl}. ORCID: 0000-0003-3170-6850}}

\begin{document}
\maketitle

\begin{abstract} 
Since their introduction by Karo\'nski,  Scheinerman,  and Singer--Cohen in 90's,  random intersection graphs have attracted much attention,  mostly due to their various applications. In this article,  we establish a general result for the probability that a given small subgraph is present in a random intersection graph. Our theorem applies to a broad range of parameters,  including regimes that have not previously been considered in the literature. Moreover, it extends known results and situates them within a broader framework, making this approach a convenient tool for comparing the probabilities of graphs in the random intersection graph model. To prove our main result,  we develop new proof techniques that,  to the best of our knowledge,  have not previously been employed in the study of small subgraphs in random intersection graphs.   Many special cases and examples are presented as well.
\end{abstract}
\noindent\emph{MSC2020 subject classification}: Primary 05C80; Secondary 05C82.

\section{Introduction}

The random intersection graph model $\Gnmp$ was introduced in \cite{1999KarSchSin}. Let $\V$ be a set of $n$ vertices and let $\A=\{a_1,\ldots,a_m\}$ be a set of $m=m(n)$ attributes, where $m$ is nondecreasing and $p=p(n)\in(0,1)$. Independently for each pair $(v,a_i)\in\V\times\A$, vertex $v$ is assigned attribute $a_i$ with probability $p$. The random intersection graph $\Gnmp$ is the graph on $\V$ in which two distinct vertices are adjacent whenever they are attributed at least one common attribute.

Since its introduction,  the random intersection graph has gained a lot of attention due to wide applications such as ``gate matrix layout'' for VLSI design \cite{1999KarSchSin},    cluster analysis and classification \cite{RIGGodehardt1},    analysis of complex networks  \cite{RIGClustering2,    RIGTunableDegree},    secure wireless networks  \cite{WSNphase2} or epidemics \cite{GpEpidemics}. Further details can be found in the survey papers \cite{RIGsurvey1,   RIGsurvey2,  Spirakis2013}. 

One of the difficulties in studying this model,  compared with the classical binomial model $\Gnp$,  is the dependence between edges. Although a series of papers \cite{FillScheinSinger2000Equivalence, Rybarczyk2011Equivalence, KimLeeNa2018Equivalence, BrennanBresterNagaraj2020Equivalence} has shown that,  when $m$ is asymptotically larger than $n^3$,  the random intersection graph $\Gnmp$ is asymptotically equivalent to a binomial random graph whose edge probability is close to that of $\Gnmp$,  for the remaining range of $m$,  even estimating the existence probabilityof a given small subgraph  remains quite challenging.
 
The first related results appeared in \cite{1999KarSchSin},  where the notion of so--called clique covers was introduced and an asymptotic expression for the probability of a subgraph with a specified clique cover was derived under the additional assumption that $m=\lfloor n^{\alpha}\rfloor$,  where $\alpha>0$ is a constant. Although a result analogous to Theorem~\ref{Lem:CCprobability} is not stated explicitly there,  it is discussed on pages 138--139. In \cite{2017RybSta},  the arguments from \cite{1999KarSchSin} were formalized and made more rigorous,  while the assumption $m=\lfloor n^{\alpha}\rfloor$ was retained. In \cite{Dong2023},  an analogous result (Lemma~3.1 therein) was established without imposing the assumption $m=\lfloor n^{\alpha}\rfloor$. However,  the result is meaningful only in the regime $mp^2\to 0$ as $n\to\infty$.

Both assumptions,  namely $m=\lfloor n^{\alpha}\rfloor$ and $mp^2\to 0$ as $n\to\infty$,  fail to cover parameter regimes that may be of interest. Recall that $p=\sqrt{(2\ln n+c_n)/m}$ is the threshold for $\Gnmp$ to be a complete graph (see \cite{1999KarSchSin}). More precisely,  if $c_n\to\infty$,  then $\Gnmp$ is a complete graph with probability tending to $1$,  whereas if $c_n\to-\infty$,  then with probability tending to $1$ as $n\to\infty$,  there exist at least two vertices that are not adjacent. Therefore in the study of $\Gnmp$ one concentrate on the values $p\le \sqrt{(2\ln n)/m}$.
If $m=\lfloor n^{\alpha}\rfloor$ for some constant $\alpha>0$,  then,  for $p\le \sqrt{(2\ln n)/m}$,  we have $mp^3\to 0$ as $n\to\infty$. However,  if $m=m(n)\to\infty$ sufficiently slowly,  this is no longer necessarily the case. In fact,  it is possible to have $mp^3\to\infty$ even when $p$ remains below the threshold for $\Gnmp$ to be a complete graph.

Our approach allows us to derive a general result without imposing any of the aforementioned assumptions. Moreover,  we obtain sharp bounds on the corresponding error terms,  which were not provided in \cite{1999KarSchSin, 2017RybSta, Dong2023}. Last but not least,  our approach provides a general framework for the analysis and comparison of clique covers. We establish asymptotic results and estimates for clique--cover probabilities,  as well as for certain important families of such probabilities. Building on these results,  we may develop tools that enable highly accurate comparisons between clique covers with similar structures. Examples illustrating applications of these results are provided in Section~\ref{Sec:Examples}.

To prove our main theorem,  we introduce an entirely new proof method,  which allows us to substantially weaken the assumptions.


In the article we   use the following asymptotic notation:\\
 $a_n=o(b_n) \Leftrightarrow $ $a_n/b_n\to 0$  as $n\to \infty$,   \\ 
 $a_n=O(b_n) \Leftrightarrow $ $|a_n|\le Cb_n$  for some $C>0$ and all  $n\in \N$, \\  
 $a_n\lesssim b_n \text{ or }b_n\gtrsim a_n\Leftrightarrow   a_n=O(b_n)$ and $a_n,     b_n>0$,   \\ 
 $a_n\asymp b_n \Leftrightarrow   a_n\lesssim b_n$ and $ b_n\lesssim a_n$.

We also use standard notation from graph theory.  For a graph $G=(V(G),     E(G))$ we denote by $V(G)$ its vertex set and by $E(G)$ its edge set. 
For $V'\subseteq V(G)$ and $E'=\{e\in E(G): e\subseteq V'\}$,     $G[V']=(V',  E')$ is the subgraph of $G$ induced on the vertex set $V'$. 
Furthermore,  
$G\cup G'=(V(G)\cup V(G'),     E(G)\cup E(G'))$. 

\section{Main result}

The probability of existence of any edge is $\Gnmp$ given by
\begin{align*}
\pp=1-(1-p^2)^m\asymp 1\wedge (mp^2).
\end{align*}

A key tool in determining subgraph probabilities in $\Gnmp$ is the notion of clique covers, first introduced in \cite{1999KarSchSin}. In the present paper,  we follow the same  approach,  but extend it to a more general setting. In particular,  rather than estimating the probability associated with a specific clique cover,  we establish estimates for the probabilities associated with certain classes of clique covers.

Let us now introduce a precise definition of a clique cover.
Let $H$ be a fixed graph,  and define $\calP(H)$ to be the collection of all subsets of $V(H)$ that contain both endpoints of at least one edge of $H$. We say that a family $\cover\subseteq\calP(H)$ is a {\sl clique cover} of $H$ if each edge $e\in E(H)$ is contained in some $C\in\cover$. The collection of all clique covers of $H$ will be denoted by $\covers(H)$.

Suppose that $H$ is a graph with $V(H)\subseteq\V$. For an attribute $a\in\A$,  consider the set of vertices of $H$ that choose $a$. If this set is precisely $C\subseteq V(H)$,  we say that $a$ {\sl builds} $C$. Thus, 
\begin{align}\label{eq:p_C} \Pra{a\text{ builds } C}=:p_C=p^{|C|}(1-p)^{|V(H)|-|C|}. 
\end{align} 
Given a fixed graph 
$H$ with $V(G)\subseteq \V$, $H$ is said to be {\sl given by a clique cover} $\cover\subseteq\calP(H)$ in $\Gnmp$ if the subsets of $V(H)$ belonging to $\calP(H)$ that are built by attributes are precisely the members of $\cover$. In other words,  every $C\in\cover$ is built by at least one attribute,  whereas no $C\in\calP(H)\setminus\cover$ is built by any attribute. 

By $\pi(H, \cover)$ we denote the probability that $H$ is given by the clique cover $\cover$ in $\Gnmp$. More generally,  for any family $\covers\subseteq\covers(H)$,  we write
\begin{equation}\label{Eq:SubgraphProbabilityCCfamily}
	\pi(H,   \covers):=\sum_{\cover\in \covers} \pi(H,   \cover). 
\end{equation}for the probability that $H$ is given by a clique cover belonging to $\covers\subseteq\covers(H)$

In particular,   the probability that $H$ is a subgraph of $\Gnmp$ is given by
\begin{align}\label{Eq:SubgraphProb}
	\pi(H):=\Pra{H\subseteq \Gnmp}=\pi(H,   \covers(H))
	=\sum_{\cover\in\covers(H)}\pi(H,   \cover) 
\end{align}

Next,  we introduce certain families of clique covers . These families of clique covers provide a useful tool for comparing subgraph probabilities in $\Gnmp$. Furthermore, working with entire families of clique covers, rather than individual clique covers, considerably simplifies the derivation of subgraph probabilities in $\Gnmp$. We illustrate this approach in Example~\ref{Ex:Var}. For a clique cover
$\cover_+\subseteq \calP (H)$ and a set
$\cover_-\subseteq \calP(H)\setminus \cover_+$
we denote by  $\covers(H,   \cover_+,   \cover_-)$ the family of  all the clique covers $\cover$ containing all sets from  $\cover_+$ and no set from $\cover_-$,    i.e.  such that $\cover_+\subseteq \cover\subseteq\calP(H)\setminus \cover_-$. The sets from $\cover_-$ will be called {\sl forbidden}. We will be using the shorthand notation
\begin{equation}\label{Eq:SubgraphProbabilityCCfamily2}
	\pi(H,   \cover_+,   \cover_-):=\pi(H,   \covers(H,   \cover_+,   \cover_-))=\sum_{\cover\in \covers(H,   \cover_+,   \cover_-)} \pi(H,   \cover). 
\end{equation}
As a special case,  we have $\pi(H,   \cover)=\pi(H,   \cover,   \calP(H)\setminus\cover)$.

In Theorem \ref{Lem:CCprobability},  we present a general result concerning the asymptotic behaviour and estimates of the probabilities $\pi(H, \cover_+, \cover_-)$.

\begin{theorem}\label{Lem:CCprobability}
	Let $H$ be a graph and   $\calP(H)$  be the family of subsets of $V(H)$ containing both ends of at least one edge of $H$. Fix a clique cover  $\cover_+$   and a set $\cover_-\subseteq \calP(H)\setminus \cover_+$.
	If  $m\to \infty$ and $p$ is bounded away from $1$ by a constant,   then 
	\begin{align}\label{eq:CC1}
		\pi(H,  \cover_+,   \cover_-)=\left(1+O(m^{-\delta_{|V(H)|}})\right)  \left(1-\sum_{C\in \cover_-}p_C\right)^{m}\prod_{C\in\cover_+}(1-e^{-mp_C}), 
	\end{align}
	where $p_C$ is defined in \eqref{eq:p_C},  and $\delta_{k}=1/({k2^{k+1}})$. 
\end{theorem}

Importantly,  when $mp^3\le 1$,  the behaviour of described above probabilities is governed by the cardinalities of $\cover_+$,  $\cover_-$,  and certain of their subfamilies. This dependence is made precise in the following corollary. 
Let us note that the assumption $mp^3\le 1$ is not very restricting,  as for $mp^3\ge 1$
$$\pp = 1-\left(1-p^2\right)^m \ge 1-(1-m^{-2/3})^m=1-e^{-m^{1/3}+O(m^{-1/3})}=1-o(1)$$
for $m\to\infty$. On the other hand,  as mentioned in the introduction,  when $m$ tends slowly to infinity,  values $p\ge m^{-1/3}$ may still be of interest,  provided that $p=m^{-1/3}$ lies below the threshold for $\Gnmp$ to be a clique,  i.e. $m^{-1/3}=o((\ln n/m)^{1/2})$.

\begin{corollary}\label{Cor:CCprobability} Under conditions of Lemma~\ref{Lem:CCprobability}, 
	if $mp^3\le 1$ and $m>m_0$ for some $m_0\in\N$,  the following estimates hold true 
	\begin{align}\label{Eq:Cor1}
		\pi(H,  \cover_+,   \cover_-)&=\left(1+O(m^{-\delta_{|V(H)|}}+mp^3)\right)
		  \pp^{|\cover_+^{(2)}|}(1-\pp)^{|\cover_-^{(2)}|}\prod_{C\in\cover_+\setminus \cover_+^{(2)}}mp^{|C|}, 
	\end{align}
	where $\cover_+^{(2)}$ and $\cover_-^{(2)}$ are the families of sets of cardinality $2$ in $\cover_+$ and $\cover_-$,   respectively,  and $\delta_{k}=1/({k2^{k+1}})$.  In particular, for $mp^2\le1$ we have
	\begin{align}\label{Eq:Cor3}
	\pi(H,   \cover_+,    \cover_-)&\asymp \pi(H,   \cover_+)\asymp \prod_{C\in\cover_+}mp^{|C|}.
	\end{align}
\end{corollary}

The next proposition will allow us to compare the probabilities of graphs with similar,  though not identical,  structures.
Let $\cover$ and $\cover'$ be two families of sets. We write $\cover\Bumpeq\cover'$  if  there is a bijection $f:\cover\rightarrow \cover'$ such that $|f(C)|=|C|$ for any $C\in \cover$. {Applications will be shown in Examples~\ref{Ex:Var} and \ref{Ex:Star:Path}.}
\begin{proposition}\label{Lem:CCequality}
	Let $H_1$ and $H_2$ be two graphs with the same number of vertices. Let,    for $i=1,   2$,    $\cover_+^{(i)}$ be a clique cover of $H_i$ and 
	$\cover_-^{(i)}\subseteq \calP(H_i)\setminus \cover_+^{(i)}$. If $\cover_+^{(1)}\Bumpeq\cover_+^{(2)}$  and $\cover_-^{(1)}\Bumpeq\cover_-^{(2)}$,  then
	$$
	\pi(H_1,   \cover_+^{(1)},   \cover_-^{(1)})=\pi(H_2,  \cover_+^{(2)},   \cover_-^{(2)}).
	$$
\end{proposition}

In sections~\ref{Sec:Proof1}--\ref{Sec:Proof3} we provide detailed proofs of Theorem~\ref{Lem:CCprobability},  Corollary~\ref{Cor:CCprobability},  and Proposition~\ref{Lem:CCprobability}. We conclude the article with Section~\ref{Sec:Examples},  in which we present several applications of our result and provide examples that demonstrate its significance.

\subsection{Proof of Theorem \ref{Lem:CCprobability}}\label{Sec:Proof1}

Denote $|V(H)|=h$,   $|\calP(H)|=t$,   and $\calP(H)=\{C_1,  C_2,  \ldots,  C_{t}\}$. Then,  let $\Ind,   \Jind\subseteq \{1,  \ldots,   t\}$,   $\Ind\cap\Jind=\emptyset$,   
be such that  $\cover_+=\{C_i: i\in \Ind\}$ is a clique cover of $H$ and $\cover_-=\{C_i: i\in \Jind\}$ is a family of forbidden sets. Furthermore,  put 
\begin{equation}\label{Eq:pi}
	p_i=p_{C_i}=p^{|C_i|}(1-p)^{h-|C_i|},  \text{ for }i\in \{1,  \ldots,  t\}\quad\text{and, }\quad p_0=1-\sum_{i\in\Ind\cup\Jind}p_i.	
\end{equation}
Given a sequence of non negative integers $(a_i)_{i\in \Ind}$,  $\Ind\subseteq \{1,  \ldots,  t\}$,    we set $$a_0=m-\sum_{i\in \Ind}a_i, $$
and denote by  
\begin{align*} 
	\{(a_i,  i\in \Ind)^{\ge0}\}\quad&\text{ all sequences } (a_i)_{i\in\Ind}\text{ such that }\quad a_i\ge 0,   i\in \Ind\cup\{0\};\\
	\{(a_i,  i\in \Ind)^{\ge1}\}\quad&\text{ all sequences } (a_i)_{i\in\Ind}\text{ such that }\quad a_i\ge 1,   i\in \Ind\text{ and }  a_0\ge 0.
\end{align*}
Furthermore,  for any such a sequence we put
\begin{align*}
	\binom{m}{a_i,  i\in \Ind}=\frac{m!}{(1-\sum_{j\in \ind}a_j)!\prod_{i\in \Ind}a_i!}=\frac{m!}{\prod_{i\in \Ind\cup\{0\}}a_i!}.
\end{align*}
Additionally,  we  decompose
$\Ind=\Ind_1\dot\cup\Ind_2\dot\cup\Ind_3$,   where 
\begin{align*}
	\Ind_1&=\{i\in\Ind:mp_i\le  \varepsilon\}, &	
	\Ind_2&=\{i\in\Ind:\varepsilon<mp_i<\omega\}, &
	\Ind_3&=\{i\in\Ind:mp_i\ge \omega\},  
\end{align*}
with
\begin{align*}
	\varepsilon&=\varepsilon(h,  m)=m^{-\frac{1}{h2^{h+1}}},  &
	\omega&=\omega(h,  m)=m^{\frac{1}{2(h+2)}}.
\end{align*}

We are now prepared to determine  two formulae for $\pi(H,    \cover_+,   \cover_- )$. Let $M_0$ be the number of attributes from $\A$  that build none of the sets  $C_i$ in $\Gnmp$,    $i\in \Ind\cup \Jind$,    and let
$M_i$,    $i\in \Ind\cup\Jind$,    be the number of attributes from $\A$  that build~$C_i$ in $\Gnmp$. Recall that by definition \eqref{Eq:pi},   for $i\neq 0$,   $p_i$ is the probability that a given attribute builds $C_i$. Moreover,  $p_0$ is the probability that a given attribute does not build any set from $\cover_+\cup\cover_-$,  as any attribute can only build  one set from $\calP(H)$ at a time. All this allows us to write
\begin{align}\nonumber
	\pi(H,    \cover_+,   \cover_- )&=\sum_{\{(a_i,  i\in \Ind)^{\ge1}\}}\,   \Pra{\{M_0=a_0\}\cap\bigcap_{i\in \Ind}\{M_i=a_i\}\cap \bigcap_{j\in \Jind}\{M_j=0\}}\\\label{Eq:Pi2}
	&=\sum_{\{(a_i,   i\in\Ind)^{\ge 1}\}}\binom{m}{a_i,   i\in\Ind}p_0^{a_0}\prod_{i\in \Ind}p_i^{a_i}.
\end{align}
Next,  define events $A_i=\{M_i\ge 1\}$,  for $i\in \Ind$,  and $B=\{\forall_{j\in \Jind} M_j=0\}$. Moreover,  denote by $A_i'=\{M_i=0\}$ the complement of $A_i$,  $i\in \Ind$.  Therefore,  by
the inclusion-exclusion principle we have 	
\begin{align}\nonumber
	\pi(H,    \cover_+,   \cover_- )
	&=\Pra{\bigcap_{i\in \Ind}A_i\cap B}\\\nonumber
	&=\Pra{B}-\Pra{\bigcup_{i\in \Ind}(A_i'\cap B)}\\\nonumber
	&=\sum_{\Lind\subseteq\Ind}(-1)^{|\Lind|}\Pra{\bigg(\bigcap_{i\in \Lind}A_i'\bigg)\cap B}\\\label{Eq:Pi1}
	&=\sum_{\Lind\subseteq\Ind}(-1)^{|\Lind|}
	\sum_{\substack{\{(a_i,   i\in\Ind)^{\ge0}\}\\   a_i=0 \text{ for }i\in \Lind}}\binom{m}{a_i,   i\in\Ind}p_0^{a_0}\prod_{i\in \Ind}p_i^{a_i},  
\end{align}	
where,  for notational convenience,  we set $\bigcap_{i\in\emptyset}A_i'$ as the whole probability space.
Next,  by the multinomial theorem,   for any $\Lind\subseteq \Ind$ we have
\begin{multline}\label{Eq:SumazLI}
	\sum_{\substack{\{(a_i,   i\in\Ind)^{\ge0}\}\\  a_i=0 \text{ for }i\in \Lind}}\binom{m}{a_i,   i\in\Ind}p_0^{a_0}\prod_{i\in \Ind}p_i^{a_i}
	=
	\sum_{\{(a_i,  {i\in\Ind\setminus\Lind})^{\ge0}\}}\binom{m}{a_i,   i\in\Ind\setminus\Lind}p_0^{a_0}\prod_{i\in \Ind\setminus\Lind}p_i^{a_i}\\
	=\left(p_0+\sum_{i\in\Ind\setminus\Lind}p_i\right)^m
	=\left(1-\sum_{i\in\Jind\cup\Lind}p_i\right)^m.
\end{multline}
We will also consider the sums over $\{(a_i,   i\in\Ind_2\cup \Ind_3)^{\ge0}\}$. For them,   analogously as above,  we obtain
\begin{equation}\label{Eq:SumazLI23}
	\sum_{\substack{\{(a_i,   i\in\Ind_2\cup \Ind_3)^{\ge0}\}\\ a_i=0 \text{ for }i\in \Lind}}\binom{m}{a_i,   i\in\Ind}p_0^{a_0}\prod_{i\in \Ind}p_i^{a_i}
	=\left(p_0+\sum_{i\in\Ind_2\cup\Ind_3\setminus\Lind}p_i\right)^m
	=\left(1-\sum_{i\in\Jind\cup\Ind_1\cup\Lind}p_i\right)^m.
\end{equation}
Therefore,   substituting \eqref{Eq:SumazLI} to \eqref{Eq:Pi1}  we get
\begin{equation}\label{Eq:PiInclusionExclusion}
	\pi(H,   \cover_+,   \cover_-)
	=\sum_{\Lind\subseteq\Ind}(-1)^{|\Lind|}\left(1-\sum_{i\in\Jind\cup\Lind}p_i\right)^m.
\end{equation}
Similarly,   by the inclusion--exclusion principle and \eqref{Eq:SumazLI23}
\begin{equation}\label{Eq:PiInclusionExclusion2}
	\begin{split}
		&\sum_{\{(a_i,   i\in\Ind_2\cup \Ind_3)^{\ge1}\}}\binom{m}{a_i,   i\in\Ind}p_0^{a_0}\prod_{i\in \Ind}p_i^{a_i}\\
		&=\sum_{\Lind\subseteq\Ind_2\cup \Ind_3}(-1)^{|\Lind|}
		\sum_{\substack{\{(a_i,   i\in\Ind_2\cup \Ind_3)^{\ge0}\}\\  a_i=0 \text{ for }i\in \Lind}}\binom{m}{a_i,   i\in\Ind}p_0^{a_0}\prod_{i\in \Ind}p_i^{a_i}\\
		&=\sum_{\Lind\subseteq\Ind_2\cup\Ind_3}(-1)^{|\Lind|}\left(1-\sum_{i\in\Jind\cup\Ind_1\cup\Lind}p_i\right)^m.
	\end{split}
\end{equation}
Furthermore,  let us observe that for $\Lind\subseteq\Ind$ such that $\Lind\cap \Ind_3\neq \emptyset$  (i.e. when $mp_j\ge \omega$ for some $j\in\Lind$) we have	
\begin{equation}
	\label{Eq:Duzempibis}
	\begin{split}
		\left(1-\sum_{i\in\Jind\cup\Lind} p_i\right)^m
		&=\left(1-\sum_{i\in\Jind}p_i\right)^m
		\left(1-\frac{\sum_{j\in\Lind}p_j}{1-\sum_{i\in\Jind}p_i}\right)^m\\
		&\le \left(1-\sum_{i\in\Jind}p_i\right)^m\left(1-\sum_{j\in\Lind}p_j\right)^m\\
		&\le\left(1-\sum_{i\in\Jind}p_i\right)^me^{-\omega}.\\
	\end{split}
\end{equation}
In the last line we have used the fact that $(1-x-y)^m\le e^{-mx-my}\le e^{-mx}$ for $0\le x+y<1$,    $x,  y\ge 0$. 
\medskip

Now,  let us return to the analysis of \eqref{Eq:PiInclusionExclusion}. In the following,  we will demonstrate that
\begin{equation}\label{Eq:Serduszko}
	\sum_{\Lind\subseteq\Ind_2\cup \Ind_3}(-1)^{|\Lind|}\left(1-\sum_{i\in\Jind\cup\Lind}p_i\right)^m =  \left(1-\sum_{i\in\Jind}p_i\right)^m\prod_{i\in\Ind_2\cup\Ind_3}(1-e^{-mp_i})
	\left(
	1+O(m^{-\frac1h})
	\right).	
\end{equation}
In all calculations we follow the convention that any product $\prod_{i\in \emptyset}$ over an empty set equals~$1$ and sum $\sum_{i\in \emptyset}$ over an empty set equals~$0$. Then all the following calculations hold true,  even when $\Ind_2=\emptyset$. We start with considering the sum over the subsets of $\Ind_2$ only;
\begin{align*}
	&\sum_{\Lind\subseteq\Ind_2}(-1)^{|\Lind|}\left(1-\sum_{i\in\Jind\cup\Lind}p_i\right)^m=\\
	&=\left(1-\sum_{i\in\Jind}p_i\right)^m
	\left(
	\sum_{\Lind\subseteq\Ind_2}(-1)^{|\Lind|}\left(1-\frac{\sum_{i\in \Lind}p_i}{1-\sum_{i\in\Jind}p_i}\right)^m
	\right)\\
	&=\left(1-\sum_{i\in\Jind}p_i\right)^m
	\Bigg[
	\sum_{\Lind\subseteq\Ind_2}(-1)^{|\Lind|}\exp\Big(-\sum_{i\in\Lind}mp_i\Big)\\
	&\hspace{1cm}+\sum_{\Lind\subseteq\Ind_2}(-1)^{|\Lind|}\left(\left(1-\frac{\sum_{i\in \Lind}p_i}{1-\sum_{i\in\Jind}p_i}\right)^m-\exp\Big(-\sum_{i\in\Lind}mp_i\Big)\right)
	\Bigg]\\
	&=\left(1-\sum_{i\in\Jind}p_i\right)^m\prod_{i\in\Ind_2}(1-e^{-mp_i})
	\left(
	1+\text{Er}_1
	\right),  
\end{align*}
where the error term $\text{Er}_1$ equals
$$\text{Er}_1=\frac{\sum_{\emptyset\subsetneq \Lind\subseteq\Ind_2}(-1)^{|\Lind|}\[\left(1-\frac{\sum_{i\in \Lind}p_i}{1-\sum_{i\in\Jind}p_i}\right)^m-\exp\left(-\sum_{i\in\Lind}mp_i\right)\]}{\prod_{i\in\Ind_2}(1-e^{-mp_i})}.$$	
Obviously $\text{Er}_1=0$ for $\Ind_2=\emptyset$. Before we proceed with technical details of estimating $\text{Er}_1$ for $\Ind_2\neq \emptyset$,   we return to \eqref{Eq:Serduszko}. By \eqref{Eq:Duzempibis} and the above equation
\begin{equation}\label{Eq:SumaInd23jeden}
	\begin{split}
		&\sum_{\Lind\subseteq\Ind_2\cup \Ind_3}(-1)^{|\Lind|}\left(1-\sum_{i\in\Jind\cup\Lind}p_i\right)^m=\\
		&=\sum_{\Lind\subseteq\Ind_2}(-1)^{|\Lind|}\left(1-\sum_{i\in\Jind\cup\Lind}p_i\right)^m
		+\sum_{\Lind\subseteq\Ind_2\cup \Ind_3,   \Lind\cap\Ind_3\neq\emptyset}(-1)^{|\Lind|}\left(1-\sum_{i\in\Jind\cup\Lind}p_i\right)^m\\
		&\le \left(1-\sum_{i\in\Jind}p_i\right)^m\prod_{i\in\Ind_2}(1-e^{-mp_i})
		\left(
		1+\text{Er}_1
		\right) + 2^h\left(1-\sum_{i\in\Jind}p_i\right)^me^{-\omega}
		\\
		&= \left(1-\sum_{i\in\Jind}p_i\right)^m\prod_{i\in\Ind_2}(1-e^{-mp_i})
		\left(
		1+\text{Er}_1+\text{Er}_2
		\right),   
	\end{split}
\end{equation}
where 
$$
\text{Er}_2=\frac{2^he^{-\omega}}{\prod_{i\in\Ind_2}(1-e^{-mp_i})}.
$$
Similarly we get the lower bound
\begin{equation}\label{Eq:SumaInd23dwa}
	\sum_{\Lind\subseteq\Ind_2\cup \Ind_3}(-1)^{|\Lind|}\left(1-\sum_{i\in\Jind\cup\Lind}p_i\right)^m\ge  \left(1-\sum_{i\in\Jind}p_i\right)^m\prod_{i\in\Ind_2}(1-e^{-mp_i})
	\left(
	1+\text{Er}_1-\text{Er}_2
	\right).
\end{equation}

Now we estimate the value of $|\text{Er}_1|$ for $\Ind_2\neq \emptyset$. Since $\Ind_2\neq \emptyset$,   there exists $i_0\in \Ind_2$ for which,   by definition of $\Ind_2$,   we have $mp_{i_0}\le \omega$,   i.e. $p_{i_0}\le \omega/m$.  Since we assume $p<1-c$,   the definition \eqref{Eq:pi} of $p_i $ gives us $p_{i_0}\ge p^h(1-p)^h$ and
\begin{equation}\label{Eq:pipi0omega}
	\forall_{i\in \Ind\cup \Jind}\quad  p_i\le p^2\le c^{-2}(1-p)^2p^2 = c^{-2}(p^h(1-p)^h)^{2/h}\le c^{-2}p_{i_0}^{\frac{2}{h}}\le c^{-2}\left(\frac{\omega}{m}\right)^{\frac{2}{h}}.
\end{equation}
Denote
$$
x=\sum_{i\in \Lind}p_i\quad\text{ and }\quad y=\sum_{i\in\Jind}p_i.
$$
By the definition of $\Ind_2$,   \eqref{Eq:pipi0omega},   and a trivial bound $|\Ind\cup \Jind|\le 2^h$,   we get for $\emptyset\subsetneq \Lind\subseteq \Ind_2$ 
\begin{equation}\label{Eq:mxixplusy}
	mx=\sum_{i\in\Lind}mp_i\le 2^h\omega\quad \text{ and }\quad x+y=\sum_{i\in\Lind\cup\Jind}p_i\le 2^hc^{-2}\omega^{\frac{2}{h}}m^{-\frac{2}{h}}.
\end{equation}
Moreover,   as the empty set is neither in a clique cover nor a forbidden set
\begin{equation}\label{Eq:1minusxy}
	1-x-y \ge (1-p)^h \ge c^{h}.
\end{equation}
We use the facts that $e^{-u/(1-u)}\le 1-u$,   for $u\in(0,  1)$,   and $1-u\le e^{-u}$ for any $u$. Moreover $0<x/(1-y)<1$ as $x+y<1$,   $x>0$,   and $y < 1$ (for $\Lind\neq \emptyset$). Therefore
\begin{align*}
	\exp(-mx)\geq \left(1-\frac{x}{1-y}\right)^m&\geq \exp\left({-\frac{mx}{(1-x-y)}}\right)\\
	&=\exp\left({-mx}\right)\exp\left({-\frac{mx(x+y)}{1-x-y}}\right)\\
	&\geq \exp(-mx)\left(1-\frac{mx(x+y)}{1-x-y}\right)\\
	&\geq  \exp(-mx)-\frac{2^{2h}\omega^{\frac{h+2}{h}}}{m^{\frac{2}{h}}c^{h+2}}, 
\end{align*}
where in the last line we used \eqref{Eq:mxixplusy} and \eqref{Eq:1minusxy}.

Using this and the inequality $1-e^{-u}\geq u(1-u),   u\geq0$,   we get for $m\geq2$
\begin{align*}
	|\text{Er}_1|&\le \frac{\sum_{\emptyset\subsetneq\Lind\subseteq\Ind_2}\[\exp\left(-\sum_{i\in\Lind}mp_i\right)-\left(1-\frac{\sum_{i\in \Lind}p_i}{1-\sum_{i\in\Jind}p_i}\right)^m\]}{\prod_{i\in\Ind_2}(1-e^{-mp_i})}\\[5pt]
	&\le \frac{2^h}{(1-e^{-\varepsilon})^{|\Ind_2|}}\cdot\frac{2^{2h}\omega^{\frac{h+2}{h}}}{m^{\frac{2}{h}}c^{h+2}}
	\le \frac{2^{3h}}{({1-\varepsilon})^{2^h}c^{h+2}}\cdot\frac{\omega^{\frac{h+2}{h}}}{m^{\frac{2}{h}}\varepsilon^{2^h}}\\
	&\le\frac{2^{3h}}{({1-2^{-1/h2^{h+1}}})^{2^h}c^{h+2}}\,  m^{-\frac1h}=:C_{h,  c}\,  m^{-\frac1h}.
\end{align*}
In addition,   for large $m$,  
\begin{align*}
	\text{Er}_2=\frac{2^he^{-\omega}}{\prod_{i\in\Ind_2}(1-e^{-mp_i})}
	\le \frac{2^h e^{-m^{\frac{1}{2(h+2)}}}}{(1-\varepsilon)^{2^h}\varepsilon^{2^h}}
	\le \frac{2^h}{({1-2^{-1/h2^{h+1}}})^{2^h}}\cdot \frac{e^{-m^{\frac{1}{2(h+2)}}}}{m^{-\frac{1}{2h}}}\le m^{-\frac1h}.
\end{align*}
Moreover,   note that for $i\in \Ind_3$ we have $e^{-mp_i}\le e^{-m^{\frac{1}{2(h+2)}}}\ll m^{-\frac1h}$,  and hence
$$
1=\prod_{i\in\Ind_3}(1-e^{-mp_i})\left(1+o(m^{-\frac1h})\right).
$$
We apply the above equality and the bounds $\text{Er}_1=O(m^{-1/h})$,   $\text{Er}_2=O(m^{-1/h})$ to \eqref{Eq:SumaInd23jeden} and \eqref{Eq:SumaInd23dwa},  which leads to  \eqref{Eq:Serduszko}.
Then \eqref{Eq:Serduszko} combined with \eqref{Eq:PiInclusionExclusion} implies the first assertion of Lemma~\ref{Lem:CCprobability} in the case $\Ind_1=\emptyset$.

Now we focus on the case $\Ind_1\neq \emptyset$.
Let $|\Ind_1|=k\geq 1$. 
Before we start,   we remark that by \eqref{Eq:PiInclusionExclusion2},   if we repeat all the calculations leading to \eqref{Eq:Serduszko},   we get that for any constant $k$ and some $\text{Er}_0=O(m^{-\frac1h})$

\begin{equation}\label{Eq:SumaInd23cztery}
	\begin{split}
		\sum_{\{(a_i,   i\in\Ind_2\cup \Ind_3)^{\ge1}\}}&\binom{m-k}{a_i,   i\in\Ind}p_0^{a_0}\prod_{i\in \Ind}p_i^{a_i}\\
		&=\sum_{\Lind\subseteq\Ind_2\cup \Ind_3}(-1)^{|\Lind|}\left(1-\sum_{i\in\Jind\cup\Ind_1\cup\Lind}p_i\right)^{m-k}\\
		&=\left(1-\sum_{i\in\Jind\cup\Ind_1}p_i\right)^{m-k}\prod_{i\in\Ind_2\cup\Ind_3}(1-e^{-(m-k)p_i})
		\left(
		1+\text{Er}_0
		\right).
	\end{split}
\end{equation}
We start with the lower bound on $\pi(H,   \cover_+,   \cover_-)$. By \eqref{Eq:Pi2}
\begin{equation}\label{Eq:I1lower}
	\begin{split}
		\pi(H,   \cover_+,   \cover_-)
		&=\sum_{\{(a_i,   i\in\Ind)^{\ge 1}\}}\binom{m}{a_i,   i\in\Ind}p_0^{a_0}\prod_{i\in \Ind}p_i^{a_i}\\
		&\ge \sum_{\substack{\{(a_i,   i\in\Ind)^{\ge1}\}\\   a_i=1 \text{ for }i\in \Ind_1}}\binom{m}{a_i,   i\in\Ind}p_0^{a_0}\prod_{i\in \Ind}p_i^{a_i}\\
		&=\left((m)_k\prod_{i\in \Ind_1}p_i\right)\cdot\left( \sum_{\{(a_i,   i\in\Ind_2\cup\Ind_3)^{\ge 1}\}}\binom{m-k}{a_i,   i\in\Ind_2\cup\Ind_3}p_0^{a_0}\prod_{i\in \Ind}p_i^{a_i}\right).
	\end{split}
\end{equation}	
We analyse the first factor of the product from the last line of \eqref{Eq:I1lower}. For large $m$,  
\begin{equation}\label{Eq:I1lowerfirst}
	(m)_k\prod_{i\in \Ind_1}p_i
	\ge \left(1-km^{-1}\right)\prod_{i\in \Ind_1}mp_i
	\ge \left(1-m^{-\frac1h}\right)\prod_{i\in \Ind_1}\left(1-e^{-mp_i}\right).
\end{equation}
The value of the other factor is estimated using \eqref{Eq:SumaInd23cztery}
\begin{align}\label{Eq:I1lowersecond}
	&\sum_{\{(a_i,   i\in\Ind_2\cup\Ind_3)^{\ge 1}\}}\binom{m-k}{a_i,   i\in\Ind_2\cup\Ind_3}p_0^{a_0}\prod_{i\in \Ind}p_i^{a_i}
	=\\ \nonumber
	&=\left(1-\sum_{i\in\Jind\cup\Ind_1}p_i\right)^{m-k}\prod_{i\in\Ind_2\cup\Ind_3}(1-e^{-(m-k)p_i})(1+\text{Er}_0)\\
	&\ge \nonumber
	\left(1-\sum_{i\in\Jind\cup\Ind_1}p_i\right)^{m}
	\prod_{i\in\Ind_2\cup\Ind_3}(1-e^{-mp_i})
	\left(1-\frac{e^{-mp_i}(e^{kp_i}-1)}{1-e^{-mp_i}}\right)(1+\text{Er}_0).
\end{align}
Let us remind that $\Ind_1\neq \emptyset$. Therefore there exists $i_0\in \Ind_1$ for which,   by definition,   $mp_{i_0}\le \varepsilon\le 1$,   i.e. $$p^h(1-p)^h\le p^{|C_{i_0}|}(1-p)^{h-|C_{i_0}|}=p_{i_0}\le m^{-1}.$$ 
Since also $p<1-c$ for some constant $c$ then,  then  we have 
$p\le c^{-1}m^{-\frac1h}$ and 
\begin{equation}\label{Eq:pimale}
	\forall_i\ p_i\le p^2\le c^{-2}m^{-\frac2h} = o(1).
\end{equation} 
Thus,   as $e^u-1\le 2u$ for $u\in (0,  1)$ and $e^u-1\ge u$,   for constant $k$ and large $m$,   we have
\begin{equation}\label{Eq:I1lowersecondbis}
	\frac{e^{-mp_i}(e^{kp_i}-1)}{1-e^{-mp_i}}=\frac{e^{kp_i}-1}{e^{mp_i}-1}\le \frac{2kp_i}{mp_i}\le \frac{k}{m} \ll m^{-\frac1h}.
\end{equation}
Moreover,  using the definition of $\Ind_1$ and \eqref{Eq:pimale}
\begin{equation}\label{Eq:I1lowersecondbisbis}
	\begin{split}
		\frac{\left(1-\sum_{i\in\Jind\cup\Ind_1}p_i\right)^{m}}
		{\left(1-\sum_{i\in\Jind}p_i\right)^m}
		&=\left(
		1-\frac{\sum_{i\in \Ind_1}p_i}{1-\sum_{i\in \Jind}p_i}
		\right)^m\\
		&\ge 1-\frac{\sum_{i\in \Ind_1}mp_i}{1-\sum_{i\in \Jind}p_i}
		\ge 1-\frac{k\varepsilon}{1-2^hc^{-2}m^{-\frac2h}}=1-O\left(m^{-\frac{1}{h2^{h+1}}}\right).
	\end{split}
\end{equation}
Combining \eqref{Eq:I1lower},   \eqref{Eq:I1lowerfirst},   \eqref{Eq:I1lowersecond},   \eqref{Eq:I1lowersecondbis},   and \eqref{Eq:I1lowersecondbisbis},   we get
\begin{equation}\label{Eq:I1lowerFinal}
	\pi(H,   \cover_+,   \cover_-)\ge \left(1+O\left(m^{-\frac{1}{h2^{h+1}}}\right)\right)\left(1-\sum_{i\in\Jind}p_i\right)^{m}
	\prod_{i\in\Ind}(1-e^{-mp_i}).
\end{equation}
Now we show the upper bound. First we need an auxiliary fact. For convenience,  let us assume that $\Ind_1=\{1,  \ldots,  k\}$ and $\Ind_2\cup\Ind_3=\{k+1,  \ldots,  t_0\}$.
Given a sequence $(a_1,  \ldots,  a_{t_0})$ from $\{(a_i,   i\in\Ind)^{\ge 1}\}$ we have
\begin{align*}
	&\binom{m}{a_i,   i\in\Ind}p_0^{m-\sum_{j\in\Ind}a_j}\prod_{i\in \Ind}p_i^{a_i}=\\
	&=
	\binom{m}{a_1}\binom{m-a_1}{a_2}\ldots\binom{m-\sum_{j=1}^{k-1}a_j}{a_k}
	\left(\prod_{i\in \Ind_1}\left(\frac{p_i}{p_0}\right)^{a_i}\right)\\
	&\quad\cdot 
	\frac
	{\left(m-\sum_{j=1}^{k}a_j\right)!}
	{\left(m-\sum_{j=1}^{k}a_j-\sum_{j\in \Ind_2\cup\Ind_3}a_j\right)!\cdot\prod_{i\in \ind_2\cup \ind_3}a_i!}p_0^{m-\sum_{j\in \Ind_2\cup\Ind_3}a_j}\prod_{i\in \Ind_2\cup\Ind_3}p_i^{a_i}\\
	&\le
	\left(\prod_{i\in \Ind_1}\binom{m}{a_i}\left(\frac{p_i}{p_0}\right)^{a_i}\right)
	\frac
	{ m!}
	{\left(m-\sum_{j\in \Ind_2\cup\Ind_3}a_j\right)!\cdot\prod_{i\in \ind_2\cup \ind_3}a_i!}p_0^{m-\sum_{j\in \Ind_2\cup\Ind_3}a_j}\prod_{i\in \Ind_2\cup\Ind_3}p_i^{a_i}
	\\
	&=
	\left(\prod_{i\in \Ind_1}\binom{m}{a_i}\left(\frac{p_i}{p_0}\right)^{a_i}\right)
	\binom{m}{a_i,  i\in \Ind_2\cup\Ind_3}
	p_0^{m-\sum_{j\in \Ind_2\cup\Ind_3}a_j}\prod_{i\in \Ind_2\cup\Ind_3}p_i^{a_i}	.
\end{align*}
Since  $\{(a_i,   i\in\Ind)^{\ge 1}\}\subseteq \{1,  \ldots,  m\}^{t_0}\times \{(a_i,   i\in\Ind_2\cup\Ind_3)^{\ge 1}\}$,  
from \eqref{Eq:Pi2} we greedily  get
\begin{equation}\label{Eq:Upper1}
	\begin{split}
		&\pi(H,   \cover_+,   \cover_-)
		=\sum_{\{(a_i,   i\in\Ind)^{\ge 1}\}}\binom{m}{a_i,   i\in\Ind}p_0^{a_0}\prod_{i\in \Ind}p_i^{a_i}\\[5pt]
		&\le\left(\prod_{i\in \Ind_1}
		\sum_{a_i=1}^{m}
		\binom{m}{a_i}
		\left(
		\frac{p_i}{p_0}
		\right)^{a_i}
		\right) \cdot
		\left(\sum_{\{(a_i,   i\in\Ind_2\cup\Ind_3)^{\ge 1}\}}\binom{m}{a_i,  i\in\Ind_2\cup\Ind_3}
		p_0^{a_0}\prod_{i\in \Ind_2\cup \Ind_3}p_i^{a_i}\right).
	\end{split}
\end{equation}
We upper bound the first factor of the product in the last line of \eqref{Eq:Upper1}.
\begin{equation}\label{Eq:Upper15}
	\prod_{i\in \Ind_1}
	\left(
	\sum_{a_i=1}^{m}
	\binom{m}{a_i}
	\left(
	\frac{p_i}{p_0}
	\right)^{a_i}
	\right)
	=\prod_{i\in \Ind_1}
	\left(\left(1+\frac{p_i}{p_0}\right)^m-1\right)
	\le 	\prod_{i\in \Ind_1}
	\left(e^{\frac{mp_i}{p_0}}-1\right).		
\end{equation}
Recall that,   for $i\in \Ind_1$,   $mp_i\le \varepsilon = m^{-\frac{1}{h2^{h+1}}}= o(1)$. Then,  if $\Ind_1\neq \emptyset$,  by \eqref{Eq:pimale} we get
$$
p_0=1-\sum_{i\in \Ind\cup\Jind}p_i\ge 1-2^hc^{-2}m^{-\frac2h} = 1-o(m^{-\frac1h}).
$$
Since $e^u\le 1 + 2u$,   for $u\in (0,  1)$,   and $u\le e^u-1$,   for all $u$,  for $i\in\Ind_1$ and large $m$ 
\begin{align*}
	e^{\frac{mp_i}{p_0}}-1
	&=
	e^{mp_i}e^{mp_i\frac{1-p_0}{p_0}}-1
	\le 
	e^{mp_i}\left(1+2mp_i\frac{1-p_0}{p_0}\right)-1=\\
	&=
	e^{mp_i}-1+mp_ie^{mp_i}\cdot 2\frac{1-p_0}{p_0}\\
	&\le 
	(e^{mp_i}-1)\left(1+3\frac{1-p_0}{p_0}\right)=(e^{mp_i}-1)(1+o(m^{-\frac1h})).
\end{align*} 
Substituting this to \eqref{Eq:Upper15} we obtain
\begin{align}
	\prod_{i\in \Ind_1}
	\left(
	\sum_{a_i=1}^{m}
	\binom{m}{a_i}
	\left(
	\frac{p_i}{p_0}
	\right)^{a_i}
	\right)
	&\le
	\left(1+o(m^{-\frac1h})\right) 
	\prod_{i\in \Ind_1}
	\left(
	e^{mp_i}-1
	\right)
	\\\label{Eq:Upper2}&=
	\left(1+o(m^{-\frac1h})\right) 
	\prod_{i\in \Ind_1}
	\left(1-
	e^{-mp_i}
	\right)
	e^{\sum_{i\in \Ind_1}mp_i}.
\end{align}
Now we find the upper bound for the second term of the product from the last line of \eqref{Eq:Upper1}. By \eqref{Eq:SumaInd23cztery}
\begin{equation}\label{Eq:Upper3}
	\begin{split}
		&\sum_{\{(a_i,   i\in\Ind_2\cup\Ind_3)^{\ge 1}\}}\binom{m}{a_i,  i\in\Ind_2\cup\Ind_3}
		p_0^{a_0}\prod_{i\in \Ind_2\cup \Ind_3}p_i^{a_i}\\
		&=\left(
		1+O\left(m^{-\frac1h}\right)
		\right)
		\left(1-\sum_{i\in\Jind\cup\Ind_1}p_i\right)^{m}\prod_{i\in\Ind_2\cup\Ind_3}(1-e^{-mp_i})\\
		&\le\left(
		1+O\left(m^{-\frac1h}\right)
		\right)
		\left(1-\sum_{i\in\Ind_1}p_i\right)^{m}\left(1-\sum_{i\in\Jind}p_i\right)^{m}\prod_{i\in\Ind_2\cup\Ind_3}(1-e^{-mp_i})
		\\
		&\le\left(
		1+O\left(m^{-\frac1h}\right)
		\right)
		e^{-\sum_{i\in\Ind_1} mp_i}\left(1-\sum_{i\in\Jind}p_i\right)^{m}\prod_{i\in\Ind_2\cup\Ind_3}(1-e^{-mp_i}).
	\end{split}
\end{equation}
Therefore,  \eqref{Eq:Upper1},   \eqref{Eq:Upper2}  and \eqref{Eq:Upper3} imply
$$
\pi(H,   \cover_+,   \cover_-)\le \left(1+O\left(m^{-\frac{1}{h}}\right)\right)\left(1-\sum_{i\in\Jind}p_i\right)^{m}.
\prod_{i\in\Ind}(1-e^{-mp_i}).
$$  
The upper bound matches the lower bound from \eqref{Eq:I1lowerFinal},  which completes the proof of Theorem~\ref{Lem:CCprobability}.

\subsection{Proof of Corollary~\ref{Cor:CCprobability}}

Let $mp^3\le 1$. Then $p=O(m^{-1/3})$. Note that by the formula \eqref{eq:p_C} and the inequality  $(1-p|V(H)|)p^{|C|}\le p_C\le p^{|C|}$ we have
\begin{align}\label{eq:p_Cest}
	p_C=p^{|C|}(1+O(p)) \quad\text{ and }\quad p_C\le p^2.
\end{align}

We will be also using the standard expansions  $1-x=e^{-x(1+O(x^2))}$ (equivalently $\ln (1-x)=-x+O(x^2)$) and  $1-e^{-x}=x(1+O(x))$   valid for $x\in[0, 1/2]$.

For $mp^3\le1$,  by \eqref{eq:p_Cest} we get
\begin{align}
	\left(1-\sum_{C\in \cover_-}p_C\right)^{m}&=\left(1-|\cover_-^{(2)}|p^2-O(p^3)\right)^{m}=\exp(-|\cover_-^{(2)}|mp^2+O(mp^3))\\
	&= \exp\left(|\cover_-^{(2)}|m\ln(1-p^2)+O(mp^4+mp^3)\right)\\\label{eq:1-pc}
		&= (1-\pp)^{|\cover_-^{(2)}|}\left(1+O(mp^3)\right).
\end{align}
Furthermore,  since 
\begin{align*}
		1-e^{-mp^2(1-p)^{|V(H)|-2}}&=1-e^{m\ln(1-p^2)+O(mp^3)}=1-e^{m\ln(1-p^2)}(1+O(mp^3))\\
		&=1-(1-p^2)^m+O(mp^3)=\pp+O(mp^3), 
	\end{align*}
	we obtain
	\begin{align*}
	\prod_{C\in\cover_+}(1-e^{-mp_C})&=\prod_{C\in\cover_+^{(2)}}(1-e^{-mp_C})\prod_{C\in\cover_+\setminus \cover_+^{(2)}}\left(mp^{|C|}(1+O(mp^3+p))\right)\\
	&=(\pp+O(mp^3))^{|\cover_+^{(2)}|}\left(1+O(mp^3+p)\right)\prod_{C\in\cover_+\setminus \cover_+^{(2)}}mp^{|C|}\\
	&=\left(1+O(\pp^{-1}mp^3+p)\right)\pp^{|\cover_+^{(2)}|}\prod_{C\in\cover_+\setminus \cover_+^{(2)}}mp^{|C|}.
\end{align*}
Applying this and \eqref{eq:1-pc} to \eqref{eq:CC1} we arrive at 
\begin{align*}
		\pi(H,  \cover_+,   \cover_-)=\left(1+O(m^{-\delta_{|V(H)|}}+\pp^{-1}mp^3+p)\right)  (1-\pp)^{|\cover_-^{(2)}|}\pp^{|\cover_+^{(2)}|}\prod_{C\in\cover_+\setminus \cover_+^{(2)}}mp^{|C|}.
\end{align*}
Eventually,  for $mp^3\le 1$ we may bound
\begin{align*}
\pp^{-1}mp^3+p\asymp \frac{mp^3}{1\wedge (mp^2)}+p\le mp^3+2p\le mp^3+2m^{-1/3} \le mp^3+2m^{-\delta_{|V(H)|}}, 
\end{align*}
which completes the proof.

\subsection{Proof of Proposition \ref{Lem:CCequality}}\label{Sec:Proof3}
By the definition of $p_i$ \eqref{Eq:pi} and the equality \eqref{Eq:Pi2},   for any $H$,    the value of $\pi(H,    \cover_+,   \cover_- )$ depends only on the size of its vertex set and on the cardinalities of the sets in $\cover_+$ and $\cover_-$.  Thus,   Lemma~\ref{Lem:CCequality} follows.

\section{Applications and discussion}\label{Sec:Examples}

Our results are useful for estimating the moments of random variables counting small subgraphs. Let ${\bf H}$ be a graph and $X_{\bf H}$ be the random variable that counts the number of subgraphs isomorphic to ${\bf H}$ in $\Gnmp$. Let $\cal{H}$ be family of all graphs   isomorphic to ${\bf H}$ with the vertex set in $\V$. Then
\begin{equation*}
	\mathbb{E} X_{\bf H} = \sum_{H\in {\cal H}} \pi(H), \quad
	\operatorname{Var} X_{\bf H}= \sum_{H_1, H_2\in {\cal H}}\left(\pi(H_1\cup H_2)-\pi(H_1)\pi(H_2)\right).
\end{equation*}

In what follows we give,  as examples,  some tight bounds for $\pi(H)$ in some special cases. Moreover,  in Example~\ref{Ex:Var} and Example~\ref{Ex:Star:Path} we show  how Corollary~\ref{Lem:CCequality} combined with Theorem~\ref{Lem:CCprobability} allows us to obtain sharp bounds on the differences similar to those in the term $\pi(H_1\cup H_2)-\pi(H_1)\pi(H_2)$.

We begin with several elementary results concerning subgraph probabilities $\pi(H)$. Some of them are in the spirit of the results on induced subgraphs presented in \cite{1999KarSchSin},  although they were not explicitly stated there. Moreover,  \cite{1999KarSchSin} focuses primarily on the threshold function for existence of induced subgraphs. Therefore,  for the sake of completeness,  we provide a detailed discussion together with complete proofs.

Before turning to the main results, let us first make a general observation. Let $H_1$ and $H_2$ be any graphs with the vertex sets in $\V$. 	The indicators $\mathbf 1_{\{H_1\subseteq \Gnmp\}}$ and $\mathbf 1_{\{H_2\subseteq \Gnmp\}}$ are increasing in relation to Bernoulli random variables representing links between vertices and their attributes. Therefore,  by FKG inequality (see for example Theorem 2.12 in \cite{KsiazkaJLR}), we have
$$\E[\mathbf 1_{\{H_1\subseteq \Gnmp\}}\mathbf 1_{\{H_2\subseteq \Gnmp\}}]\ge \E[\mathbf 1_{\{H_1\subseteq \Gnmp\}}]\E[\mathbf 1_{\{H_2\subseteq \Gnmp\}}], $$
which gives the following fact. 

\begin{fact}\label{Lem:Division2}
		For any graphs $H_1,  H_2$ it holds that
		$$\pi(H_1\cup H_2)\geq \pi(H_1)\pi(H_2).$$
\end{fact}

This not only shows that $\pi(H_1\cup H_2)-\pi(H_1)\pi(H_2)$ is positive, but also leads to a simple estimate for $\pi(H)$. Namely, by applying Fact~\ref{Lem:Division2} inductively, treating each edge as a subgraph of $H$, we obtain  
\begin{equation*}
	\pp^{|E(H)|}\le \pi(H).
\end{equation*}

Recall that in the binomial random graph $G(n,\pp)$, the subgraph probability is given exactly by $\pp^{|E(H)|}$. In the random intersection graph model, however, the quantity $\pp^{|E(H)|}$ approximates the subgraph probability accurately only in specific ranges of the parameters. To see this, we first note that for $mp^2\ge 1$,  a careful examination of Theorem~\ref{Lem:CCprobability} for $mp^2\ge 1$ reveals that $ \pi(H)=(1+o(1))\pi(H, E(H), \emptyset)=(1+o(1))\pp^{|E(H)|}.$
In light of the above remarks,  in estimating $\pi(H)$,  we concentrate on the particularly interesting regime $mp^2 \le 1$.

Before deriving estimates for $\pi(H)$,  we note that it is plausible that,  in the sum \eqref{Eq:SubgraphProb},  only certain clique covers play a dominant role and,  in fact,  determine the value of $\pi(H)$. As already mentioned in \cite{1999KarSchSin},  it is indeed the case. First we concentrate on those subgraph probabilities  that are dominated by clique covers $E(H)$ or $\{V(H)\}$. We first provide precise estimates for the probabilities of these clique covers. Under this assumption $mp^2\le 1$ we have $mp^3\le m\cdot m^{-3/2}\le m^{-\delta_k}$ for any $k\ge 1$,  and hence \eqref{Eq:Cor1}  implies the following result.
\begin{corollary}\label{Cor:EH}
	Let $H$ be a graph and $mp^2\le 1$. Then  
	\begin{align}\label{eq:CC3}
		\pi(H,   E(H))&=(1+O(m^{-\delta_{|V(H)|}}))\pp^{|E(H)|},  \\\label{eq:CC4}
		\pi(H,  \{V(H)\})&=(1+O(m^{-\delta_{|V(H)|}}))(1-\pp)^{|E(H)|}mp^{|V(H)|}, 
	\end{align}
	\ \\[-10pt]where $\delta_{k}=1/({k2^{k+1}})$. As a consequence,  we have
		\begin{align*}
			\pi(H,   E(H))\asymp \pp^{|E(H)|},  \quad
			&\pi(H,  \{V(H)\})\asymp mp^{|V(H)|}.
	\end{align*}
\end{corollary}

To determine the value of $\pi(H)$, it is crucial to understand which clique cover dominates in a given parameter regime. Below, we address this question for the clique covers $\{V(H)\}$ and $E(H)$.
\begin{proposition}\label{Prop:V}
	Let $mp \le 1$ and $H$ be a connected graph. Then
	\begin{align*}
		\pi(H) &= \left(1+O(mp)\right)\pi(H,  \{V (H)\})\\&= \left(1+O(m^{-\delta_{|V(H)|}}+mp)\right)mp^{|V (H)|}, 
	\end{align*}
	where $\delta_{k}=1/({k2^{k+1}})$.
\end{proposition}
\begin{proof}
	For $H$ being an edge,  there is only one clique cover $\E(H)=\{V(H)\}$,  and the assertion follows trivially by Theorem \ref{Lem:CCprobability}. Let $|E(H)|\ge 2$ and  $\cover=\{C_1,  \ldots,  C_k\}\neq \{V(H)\}$ be a clique cover of $H$. Since  $H$ is connected,   $\cover\neq \{V(H)\}$ and $|E(H)|\geq2$,    we may assume that for some $2\le c\le |E(H)|$ the sequence  $S_i:=C_1\cup\ldots \cup C_i$,  $1\le i\le c$,  is a  sequence of increasing connected sets with $S_c=V(H)$. Then $\sum_{i=1}^c|C|\ge |V(H)|+c-1$ and,  as $mp^3\le mp\le 1$ and $1-\pp\le e^{-mp^2}\asymp1$,  \eqref{Eq:Cor1}  gives us 
	\begin{align*}
		\pi(H,  \cover)&\asymp \prod_{i=1}^c mp^{|C_i|}\overbrace{\prod_{i=c+1}^k mp^{|C_i|}}^{\le1}\\
		&\le m^cp^{|V(H)|+c-1}=mp^{|V(H)|}(mp)^{c-1}\le mp^{|V(H)|}mp.
	\end{align*}
	Combining this with \eqref{eq:CC4},  the asymptotics  $1-\pp=1-O(mp^2)$ for $mp^2\le 1$,  and \eqref{Eq:SubgraphProb} we obtain the assertion.
\end{proof}

\begin{proposition}\label{Prop:E}
	Let  $H$ be a graph with $h=|V(H)|\ge 2$. Then for $mp^2\le 1$ and $m^{1+\frac1h}p^2\ge1$ we have 
	\begin{align*}
	\pi(H)&=\left(1+O((m^{1+\frac1h}p^2)^{-3/2})\right)\pi(H,   E(H))\\
	&=\left(1+O((m^{1+\frac1h}p^2)^{-3/2}+m^{-\delta_h})\right)\pp^{|E(H)|},   
	\end{align*}
	where $\delta_h=1/({h2^{h+1}})$.
\end{proposition}
\begin{proof}For $h=2$ we have $\pi(H)=\pp$ and the assertion trivially holds, so we assume $h\ge3$.
Let $\cover=\{C_1, \ldots,  C_k\}$ be a clique cover of $H$. Let $e_i$ denote the number of edges covered by $C_i$,  i.e.,  $e_i=|E(H[C_i])|$. Then  $e_i\le \binom{|C_i|}{2}$ and $\sum_{i=1}^{t}e_i\ge |E(H)|$. For $|C_i|=2$ it clealy holds that  $mp^{|C_i|}= (mp^2)^{e_i}$. Furthermore,  since $mp^2\le 1$ and $m^{1+\frac{1}{h}}p^2\ge 1$, for $3\le k:=|C_i|\le h=|V(h)|$ we obtain
\begin{align}\nonumber
	mp^{|C_i|}
	&=\frac{mp^{k}}{(mp^2)^{\binom{k}{2}}}(mp^2)^{\binom{k}{2}}
	=\frac{(mp^2)^{\binom{k}{2}}}{\left(m^{1+\frac{1}{k}}p^2\right)^{k(k-2)/2}}\le \frac{(mp^2)^{e_i}}{\left(m^{1+\frac{1}{h}}p^2\right)^{k(k-2)/2}}\\
	&\le \frac{(mp^2)^{e_i}}{\left(m^{1+\frac{1}{h}}p^2\right)^{3/2}}.\label{Eq:Split}
\end{align}
 For $\cover\neq E(H)$ we there is at least one $C_i\in\cover$ such that $|C_i|\ge 3$. Thus, by  corollaries~\ref{Cor:CCprobability}  and \ref{Cor:EH},
\begin{align*}
	\pi(H,  \cover)&\asymp \prod_{i=1}^t mp^{|C_i|}\le \left(m^{1+\frac{1}{h}}p^2\right)^{-3/2}\prod_{i=1}^{t} (mp^2)^{e_i} \le \left(m^{1+\frac{1}{h}}p^2\right)^{-3/2} (mp^2)^{|E(H)|} \\
	&= O\left(\left(m^{1+\frac{1}{h}}p^2\right)^{-3/2}\right)\pi(H, E(H)). 
\end{align*}
The lemma follows by \eqref{Eq:SubgraphProb} and Corollary~\ref{Cor:EH}.
\end{proof}
Next, we show that for some  important classes of graphs the clique covers $E(H)$ and $\{V(H)\}$ are the leading ones in the whole range of parameters.
\begin{proposition}\label{Lem:PiOptimal}For  $H$ being a tree $\mathbf T_r$,  a cycle $\mathbf C_r$ or a complete graph $\mathbf K_r$ on $r$ vertices we have
	\begin{align}\label{eq:pi=E+V}
		\pi(H)& \asymp \pi (H,    E(H))+ \pi (H,    \{V(H)\}).\\\label{eq:pi=E+Vest}
	\end{align}
\end{proposition}
\begin{remark}
	In fact we could prove a sharper statement that for $H\in \{{\mathbf T_r},  {\mathbf C_r},  \kli_r\}$	
	\begin{equation*}
		\pi(H)=(1+o(1)) (\pi(H, \{V(H)\})+\pi(H, E(H))), \quad \text{ except the case }H={\mathbf T_r}\text{ and }mp\asymp 1.
	\end{equation*}
	However,  we omit the proof,  as it does not introduce any new ideas and merely requires considering a larger number of cases.	
\end{remark}
\begin{proof}[Proof of Proposition~\ref{Lem:PiOptimal}]
By \eqref{Eq:SubgraphProb},  in general $\pi(H)\ge \pi(H, E(H))+\pi(H, \{V(H)\})$. 
Therefore in view of previous remarks concenrning the case $mp^2\ge 1$ and Proposition~\ref{Prop:V},  it is sufficient to prove the upper bound in \eqref{eq:pi=E+V} for $mp^2\le 1$ and $mp\ge 1$.

Let $H={\mathbf T_r}$ or $H={\mathbf C_r}$.
By \eqref{Eq:SubgraphProb} 
\begin{align*}
	\pi(H)=\pi(H, \{V(H)\}, \emptyset)+\pi(H, E(H))+\sum_{\cover\in\covers(H)\setminus\{E(H)\}, V(H)\notin \cover}\pi(H, \cover)
\end{align*}
By \eqref{Eq:Cor1}, for $mp^2\le 1$ we have $\pi(H, \{V(H)\}, \emptyset)\asymp \pi(H, \{V(H)\})$. Therefore,  we need to find an upper bound on $\pi(H, \cover)$ for clique covers $\cover$ satisfying  $V(H)\notin \cover$. 
In fact it suffices  to prove that for $H={\mathbf T_r}$ and $H={\mathbf C_r}$ we have
\begin{equation}
	\forall_{\cover\in \covers(H):  V(H)\notin \cover}\  \pi(H, \cover) = O(\pi(H, E(H))).
\end{equation}
Let $\cover=\{C_1, \ldots, C_t\}$ and let $e_i$ be the number of edges covered by $C_i$,  $1\le i\le t$. 
Since $V(H)\not\in \cover$,  then $H [C_i]$ is a forest. Consequently,   we have $|C_i|=|V(H[C_i])|\geq e_i+1$,  and $\sum_{i=1}^t e_i \ge |E(H)|$. Consequently,  by \eqref{Eq:Cor3} and \eqref{eq:CC3},  for $mp\ge 1$ and $mp^2\le 1$
	\begin{align*}
		\pi(H,  \cover)
		&\asymp \prod_{i=1}^tmp^{|C_i|}
		\le  \prod_{i=1}^tmp^{e_i+1}\\
		&= \prod_{i=1}^t\frac{1}{(mp)^{e_i-1}}(mp^2)^{e_i}
		\le (mp^2)^{\sum_{i=1}^t e_i} \le (mp^2)^{|E(H)|}  =  O(\pi(H, E(H))).
	\end{align*}

	For $H=\kli_r$ we use mathematical induction. In view of Corollary \ref{Cor:EH} the assertion in equivalent to
	\begin{equation*}
		\forall_{r\ge 2}\ \ 
		\pi(\kli_{r})\asymp
		\begin{cases}
			\pi(\kli_{r}, \{V(\kli_{r})\})\asymp mp^{r}, &\text{ for }m^{-1}\le p \le m^{-\frac{r+1}{2r}};\\
			\pi(\kli_{r}, E(\kli_{r}))\asymp (mp^{2})^{\binom{r}{2}}, &\text{ for }m^{-\frac{r+1}{2r}}\le p \le m^{-\frac{1}{2}}.
		\end{cases}
	\end{equation*} 
	For $r=2$ it holds trivially. Assume that the statement holds for some $r\ge 2$. We will prove it for $r+1$.  Consider three cases. 
	
	\noindent{\sl Case 1:} $m^{-1}\le p\le m^{-\frac{r+1}{2r}}$. Then 
	\begin{equation}\label{Eq:Ind:1}
		(mp^2)^{r}\le \left(m^{1-\frac{r+1}{r}}\right)^r = m^{-1}\le p.
	\end{equation}
	Let   $v\in V(\kli_{r+1})$ and $\cover$ be a clique cover of $\kli_{r+1}$.  By \eqref{Eq:Cor3} and \eqref{Eq:Ind:1} we get
	\begin{align*}
		\pi(H,   \cover)
		&\asymp\prod_{C\in\cover^{(2)}: v\in C}mp^{2}
		\prod_{C\in\cover\setminus \cover^{(2)}: v\in C}p\cdot mp^{|C|-1}
		\prod_{C\in\cover: v\not\in C}mp^{|C|}\\
		&\le p^{|\{C\in\cover^{(2)}: v\in C\}|}
		p^{|\{C\in\cover\setminus \cover^{(2)}: v\in C\}|}
		\prod_{C\in\cover\setminus \cover^{(2)}: v\in C} mp^{|C|-1}
		\prod_{C\in\cover: v\not\in C}mp^{|C|}\\
		&\asymp p^{|\{C\in\cover: v\in C\}|}\pi\left(\kli_{r+1} - v,  \{C\in\cover: v\not\in C\}\cup\{C\setminus \{v\}: C\in\cover\setminus \cover^{(2)},  v\in C\} \right)\\
		&\le  p\,\pi(\kli_r).
	\end{align*}
Therefore, by the inductive hypothesis,
\begin{align*}
	\pi(H,   \cover)=O(p) mp^{r} = O(mp^{r+1}), 
\end{align*}
which,  together with \eqref{eq:CC4},  implies the statement in this case.

\noindent{\sl Case 2:} $m^{-\frac{r+1}{2r}}\le p\le m^{-\frac{r+2}{2(r+1)}}$. 
Using calculations analogous to those in \eqref{Eq:Split}, we obtain in this case
\begin{equation}\label{Eq:Split2}
	\forall_{k\le r}\,  mp^k\le (mp^2)^{\binom{k}{2}}\quad \text{and}\quad mp^{r+1}\ge (mp^2)^{\binom{r+1}{2}}.
\end{equation}
Let $\cover=\{C_1, \ldots, C_t\}$ be a clique cover of $\kli_{r+1}$. If $V(\kli_{r+1})\in \cover$,  then \eqref{Eq:Cor3} implies $\pi(\kli_{r+1}, \cover)=O(mp^{r+1})$. Otherwise $C_i\neq V(\kli_{r+1})$,  $i=1, \ldots, t$,  and $C_1, \ldots, C_t$ cover all the edges of $\kli_{r+1}$,  hence
\begin{equation*}
\forall_{i} |C_i|\le r\quad\text{and}\quad	\sum_{i=1}^{t}\binom{|C_i|}{2}\ge \binom{r+1}{2}.
\end{equation*} 
Therefore, by \eqref{Eq:Cor3} and \eqref{Eq:Split2},
\begin{align*}
	\pi(\kli_{r+1}, \cover)
	&\asymp \prod_{i=1}^{t}mp^{|C_i|}
	\le \prod_{i=1}^{t}(mp^2)^{\binom{|C_i|}{2}}
	\le (mp^2)^{\binom{r+1}{2}}\le mp^{r+1}.
\end{align*}
Consequently,  for every clique cover $\cover$  we have $\pi(\kli_{r+1}, \cover)=O(mp^{r+1})$,  and the assertion follows from \eqref{eq:CC4}.

\noindent{\sl Case 3:} $p\ge m^{-\frac{r+2}{2(r+1)}}$. The statement is implied by Proposition~\ref{Prop:E}.

This completes the proof.	
\end{proof}

As we can see from Proposition~\ref{Lem:PiOptimal},  in the case of trees,  cycles,  and cliques,  the behaviour of their existence probabilities is determined by the trivial clique covers: $E(H)$ and $\{V(H)\}$. Nevertheless,  this phenomenon does not extend to arbitrary subgraphs,  as demonstrated below using the butterfly graph as an example.
\begin{example}Let $H$ be a butterfly graph,  i.e.,  $H=\mathbf C^{(1)}_{3}\cup \mathbf C^{(2)}_{3}$,  where $\mathbf C^{(1)}_{3}$ and $  \mathbf C^{(2)}_{3}$ are two triangles sharing exactly one vertex. Consider the clique cover $\cover=\{V(\mathbf C^{(1)}_{3}),  V(\mathbf C^{(2)}_{3})\}$.  Then for $m^{-2/3}\ll p\ll m^{-1}$ it holds that 
	$$\pi (H)>\pi(H,  \cover)\gg \pi(H,  E(H))+\pi(H,  \{V(H)\}).$$
	Indeed,  since $p\ll m^{-2/3}$,  by \eqref{Eq:Cor3} and Corollary~\ref{Cor:EH},  we get  
	$$\pi(H,  \cover)\asymp m^2p^6=m^6p^{12}(m^{2/3}p)^{-6}\gg m^6p^{12}\asymp \pi(H,  E(H)).$$
	Furthermore,  for $p\gg m^{-1}$ we have
	$$\pi(H,  \cover)\asymp m^2p^6=mp^5(mp)\gg mp^5\asymp \pi(H,  \{V(H)\}).$$
\end{example}



Finally,  in the next two examples,  we present some applications of Proposition~\ref{Lem:CCequality}. We not only present a method for comparing the probabilities of different subgraphs,  but also illustrate the effect of edge dependencies on subgraph probabilities. The following result can also be formulated in full generality. However,  for the sake of clarity and readability,  we restrict our attention to the case $mp^3=O(1)$. We focus on this particularly interesting regime,  as it already demonstrates the relevance of formulating and proving Theorem~\ref{Lem:CCprobability} at the level of generality considered in this paper. Indeed,  although the probability of any fixed subgraph in $\Gnmp$ tends to 1 when $mp^2\to\infty$,  the probabilities associated with individual clique covers may still play an important role,  particularly when comparing the probabilities of different subgraphs. Such comparisons arise naturally in the estimation of central moments.

\begin{figure}
	\begin{center}
		\begin{tikzpicture}[xscale=1.5,   yscale=1.5]
			\tikzstyle{vert}=[line width=.3mm,    draw,    circle,    fill=white,    inner sep=2pt];
			\node[vert] (v1) at (0,   1) {};
			\node[vert] (v2) at (1,   1) {}; 
			\node[vert] (v3) at (1,   0) {};
			\node[vert] (v4) at (0,   0) {};
			\node[left] at (v1) {\footnotesize $v_1$}; 
			\node[left] at (v4) {\footnotesize $v_4$}; 
			\node[right] at (v2) {\footnotesize $v_2$}; 
			\node[right] at (v3) {\footnotesize $v_3$}; 
			\draw[ultra thick] (v1)--(v2) node[pos=.5, circle, fill=white, inner sep=1pt] {$e_1$};
			\draw[ultra thick] (v3)--(v1)node[pos=.5, circle, fill=white, inner sep=1pt] {$e_2$};
			\node at (.5, -.5) {$H_1$};
		\end{tikzpicture}
		\quad
		\begin{tikzpicture}[xscale=1.5,   yscale=1.5]
			\tikzstyle{vert}=[line width=.3mm,    draw,    circle,    fill=white,    inner sep=2pt];
			\node[vert] (v1) at (0,   1) {};
			\node[vert] (v2) at (1,   1) {}; 
			\node[vert] (v3) at (1,   0) {};
			\node[vert] (v4) at (0,   0) {};
			\node[left] at (v1) {\footnotesize $v_1$}; 
			\node[left] at (v4) {\footnotesize $v_4$}; 
			\node[right] at (v2) {\footnotesize $v_2$}; 
			\node[right] at (v3) {\footnotesize $v_3$}; 
			\draw[ultra thick] (v1)--(v2)node[pos=.5, circle, fill=white, inner sep=1pt] {$e_1'$};
			\draw[ultra thick] (v3)--(v4)node[pos=.5, circle, fill=white, inner sep=1pt] {$e_2'$};
			\node at (.5, -.5) {$H_2$};
		\end{tikzpicture}
		\quad\begin{tikzpicture}[xscale=1.5,   yscale=1.5]
			\tikzstyle{vert}=[line width=.3mm,    draw,    circle,    fill=white,    inner sep=2pt];
			\node[vert] (v1) at (0,   1) {};
			\node[vert] (v2) at (1,   1) {}; 
			\node[vert] (v3) at (1,   0) {};
			\node[vert] (v4) at (0,   0) {};
			\node[left] at (v1) {\footnotesize $v_1$}; 
			\node[left] at (v4) {\footnotesize $v_4$}; 
			\node[right] at (v2) {\footnotesize $v_2$}; 
			\node[right] at (v3) {\footnotesize $v_3$}; 
			\draw[ultra thick] (v1)--(v2) node[pos=.5, circle, fill=white, inner sep=1pt] {$e_1$};
			\draw[ultra thick] (v3)--(v1)node[pos=.5, circle, fill=white, inner sep=1pt] {$e_2$};
			\draw[ultra thick] (v1)--(v4)node[pos=.5, circle, fill=white, inner sep=1pt] {$e_3$};
			\node at (.5, -.5) {$K_{1, 3}$};
		\end{tikzpicture}
		\quad
		\begin{tikzpicture}[xscale=1.5,   yscale=1.5]
			\tikzstyle{vert}=[line width=.3mm,    draw,    circle,    fill=white,    inner sep=2pt];
			\node[vert] (v1) at (0,   1) {};
			\node[vert] (v2) at (1,   1) {}; 
			\node[vert] (v3) at (1,   0) {};
			\node[vert] (v4) at (0,   0) {};
			\node[left] at (v1) {\footnotesize $v_1$}; 
			\node[left] at (v4) {\footnotesize $v_4$}; 
			\node[right] at (v2) {\footnotesize $v_2$}; 
			\node[right] at (v3) {\footnotesize $v_3$}; 
			\draw[ultra thick] (v1)--(v2)node[pos=.5, circle, fill=white, inner sep=1pt] {$e_1'$};
			\draw[ultra thick] (v3)--(v4)node[pos=.5, circle, fill=white, inner sep=1pt] {$e_2'$};
			\draw[ultra thick] (v1)--(v4)node[pos=.5, circle, fill=white, inner sep=1pt] {$e_3'$};
			\node at (.5, -.5) {$P_4$};
		\end{tikzpicture}
	\end{center}
	\caption{\label{Fig:ExVar} Graphs from examples~\ref{Ex:Var}
		and~\ref{Ex:Star:Path}.}
\end{figure}
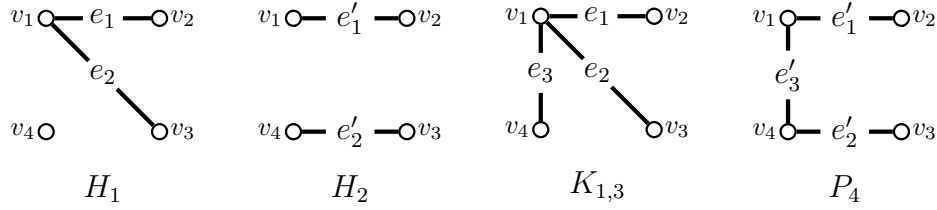

In the standard Bernoulli random graph,  edges are independent,  and hence the probability that a given path $P_3$ consisting of two edges appears as a subgraph is equal to $\pp^2$. Using Lemma~\ref{Lem:CCequality},  we can precisely quantify the difference between the probability  $\pi(P_3)$ and $\pp^2$. In this sense,  this difference provides a measure of the dependence between the two edges. Moreover,  our analysis allows us to identify the clique covers that are responsible for the dependencies between edges.

\begin{example}\label{Ex:Var}
	We will show that for $mp^3=O(1)$ and $m\to \infty$
	\begin{equation}
		\pi(P_3)-\pp^2=
		\left(1+O(m^{-\delta_4})\right)e^{-2mp^2+mp^3}\left(1-e^{-mp^3}\right),
	\end{equation}
where $\delta_k=1/({k2^{k+1}})$. In particular	
	\begin{equation}
		\pi(P_3)-\pp^2=
		\begin{cases}
			\left(1+O(m^{-\delta_4}+mp^2)\right)mp^3&\text{ for }mp^2=o(1);\\
			\left(1+O(m^{-\delta_4}+mp^3)\right)mp^3e^{-2mp^2}&\text{ for }mp^3=o(1).\\
		\end{cases}
	\end{equation}	
	
	Denote $V=\{v_1, \ldots, v_4\}$ and $A_i=V\setminus \{v_i\}$,  $i=1, 2, 3, 4$. Let $H_1$ be a graph consisting of a path $(v_2, v_1, v_3)$ and the isolated vertex $v_4$, and let $H_2$ consist of two edges $\{v_1, v_2\}$ and $\{v_3, v_4\}$ (see Figure~\ref{Fig:ExVar}). Set $e_1=\{v_1, v_2\}$ and $e_2=\{v_1, v_3\}$  in $H_1$ and  $e_1'=\{v_1, v_2\}$ and $e_2'=\{v_3, v_4\}$  in $H_2$. 
	We note that $\cP(H_1)=\{e_1, e_2, A_2, A_3, A_4, V\}$ and  $\cP(H_2)=\{e_1', e_2', A_1, A_2, A_3, A_4, V\}$.
	Moreover, by \eqref{Eq:SubgraphProb} and \eqref{Eq:SubgraphProbabilityCCfamily2},
	\begin{align}
		\nonumber\pi(P_3)&=\pi(H_1)=\\
		\label{H1:1}&=\pi(H_1, \{V\}, \emptyset)+\pi(H_1, \{e_1, e_2\}, \{V\})\\
		\label{H1:2}&+\pi(H_1, \{e_1, A_2\}, \{e_2, V\})+\pi(H_1, \{e_1, A_4\}, \{e_2, A_2, V\})\\
		\label{H1:3}&+\pi(H_1, \{e_2, A_3\}, \{e_1, V\})+\pi(H_1, \{e_2, A_4\}, \{e_1, A_3, V\})\\
		\label{H1:4}&+\pi(H_1, \{A_2, A_3\}, \{e_1, e_2, A_4, V\})\\
		\label{H1:5}&+\pi(H_1, \{A_4\}, \{e_1, e_2, V\}), 
	\end{align}   
	and
	\begin{align}
		\nonumber\pp^2&=\pi(H_2)=\\
		\label{H2:1}&=\pi(H_2, \{V\}, \emptyset)+\pi(H_2, \{e_1', e_2'\}, \{V\})\\
		\label{H2:2}&+\pi(H_2, \{e_1', A_2\}, \{e_2', V\})+\pi(H_2, \{e_1', A_1\}, \{e_2', A_2, V\})\\
		\label{H2:3}&+\pi(H_2, \{e_2', A_3\}, \{e_1', V\})+\pi(H_2, \{e_2', A_4\}, \{e_1', A_3, V\})\\
		\label{H2:4}&+\pi(H_2, \{A_2, A_3\}, \{e_1', e_2', A_4, V\})\\
	\nonumber&+\pi(H_2, \{A_1, A_4\}, \{e_1', e_2', V\})+\pi(H_2, \{A_2, A_4\}, \{e_1', e_2', A_1, V\})\\
			\label{H2:5}& + \pi(H_2, \{A_1, A_3\}, \{e_1', e_2', A_2.A_4, V\}).
	\end{align}   
	By Lemma~\ref{Lem:CCequality}
	\begin{equation}\label{Eq:Ex:Cov:1}
		\eqref{H1:1}=\eqref{H2:1},  \quad \eqref{H1:2}=\eqref{H2:2}, \quad \eqref{H1:3}=\eqref{H2:3}, \quad\text{and}\quad \eqref{H1:4}=\eqref{H2:4}.
	\end{equation}
	What has left is to comparethe expressions \eqref{Eq:Ex:Cov:2} and  \eqref{Eq:Ex:Cov:3}. By \eqref{eq:p_C} we have $p_{e_i}=p_{e_i'}=p^2(1-p)^2$,  $i=1, 2$,  $p_{A_j}=p^3(1-p)$,  $j=1, 2, 3, 4$,  and $p_V=p^4$.
	Moreover,  we will make use of the following estimates. For $mp^3=O(1)$ and a positive integer $k$ we have $1-e^{-kmp^3}\asymp mp^3$,  $mp^4=O(p)$ and $p=o(1)$.  Therefore, we get  
	\begin{equation}
	  1-e^{-kmp^3+O(mp^4)}=1-e^{-kmp^3}(1+O(mp^4))=1-e^{-kmp^3}+O(mp^4)=(1-e^{-kmp^3})(1+O(p)).
	\end{equation}
	Furthermore $p = O(m^{-1/3})$ and $1/3\ge 1/2^7 = \delta_4$. Thus
	$p=O(m^{-\delta_4})$.
	In addition,  for $x=o(1)$ we will use the relation $1+x=e^{x+O(x^2)}$.
		Therefore, by Theorem~\ref{Lem:CCprobability},  for $mp^3=O(1)$,  as $m\to \infty$, 
	\begin{align}
		\nonumber
		\eqref{H1:5}
		&=\left(1+O(m^{-\delta_4})\right)  \left(1-2p^2(1-p)^2-p^4\right)^{m}\left(1-e^{-mp^3(1-p)}\right)\\
		&=\left(1+O(m^{-\delta_4})\right) \left(1-2p^2+4p^3-3p^4\right)^m\left(1-e^{-mp^3(1-p)}\right) \nonumber\\
		&=\left(1+O(m^{-\delta_4})\right)  e^{-2mp^2+4mp^3+O(mp^4)}\left(1-e^{-mp^3(1-p)}\right)\nonumber\\
		&=\left(1+O(m^{-\delta_4})\right)  e^{-2mp^2+4mp^3+O(mp^4)}\left(1-e^{-mp^3(1-p)}\right)\nonumber\\
		&=\left(1+O(m^{-\delta_4}+mp^4+p)\right)  e^{-2mp^2+4mp^3}\left(1-e^{-mp^3}\right)
		\nonumber\\
			&=\left(1+O(m^{-\delta_4})\right)  e^{-2mp^2+4mp^3}\left(1-e^{-mp^3}\right)
		\label{Eq:Ex:Cov:2},
	\end{align}
	and similarly
	\begin{align}\nonumber 
		\eqref{H2:5}=&\left(1+O(m^{-\delta_4})\right)  \left(1-2p^2(1-p)^2-p^4\right)^{m}\left(1-e^{-mp^3(1-p)}\right)^2\\ \nonumber
		&+\left(1+O(m^{-\delta_4})\right)  \left(1-2p^2(1-p)^2-p^3(1-p)-p^4\right)^{m}\left(1-e^{-mp^3(1-p)}\right)^2\\ \nonumber
		&+\left(1+O(m^{-\delta_4})\right)  \left(1-2p^2(1-p)^2-2p^3(1-p)-p^4\right)^{m}\left(1-e^{-mp^3(1-p)}\right)^2\\ \nonumber
		&=\left(1+O(m^{-\delta_4})\right)\big(e^{-2mp^2+4mp^3+O(mp^4)}+e^{-2mp^2+3mp^3+O(mp^4)}\\ \nonumber
		&\quad\quad
		+e^{-2mp^2+mp^3+O(mp^4)}\big)
		\left(1-e^{-mp^3(1-p)}\right)^2.\\ \nonumber
		&=\left(1+O(m^{-\delta_4}+p)\right)e^{-2mp^2+4mp^3}\left(1+e^{-mp^3}+e^{-2mp^3}\right)\left(1-e^{-mp^3}\right)^2.\\ 
		&=\left(1+O(m^{-\delta_4})\right)e^{-2mp^2+4mp^3}\left(1-e^{-3mp^3}\right)\left(1-e^{-mp^3}\right).
		\label{Eq:Ex:Cov:3}
	\end{align}
	We obtain the result by subtracting \eqref{Eq:Ex:Cov:2} minus  \eqref{Eq:Ex:Cov:3}. Since $e^{-2mp^2+4mp^3}\left(1-e^{-mp^3}\right)$  is a common factor of both \eqref{Eq:Ex:Cov:2} and  \eqref{Eq:Ex:Cov:3},
	it is enough to bound the difference between the remaining terms. Note that for $mp^3=O(1)$ we have $e^{-3mp^3}\asymp 1$. Hence
	\begin{align*}
		&\left(1+O(m^{-\delta_4})\right)-\left(1+O(m^{-\delta_4})\right)\left(1-e^{-3mp^3}\right)\\
		&= 
		O(m^{-\delta_4})+e^{-3mp^3}=(1+O(m^{-\delta_4}))e^{-3mp^3}.
	\end{align*}
Thus
\begin{align*}
	\pi(P_3)-\pp^2&=\left(1+O(m^{-\delta_4})\right)  e^{-2mp^2+4mp^3}e^{-3mp^3}\left(1-e^{-mp^3}\right)\\
	&=
	\left(1+O(m^{-\delta_4})\right)  e^{-2mp^2+mp^3}\left(1-e^{-mp^3}\right).
\end{align*}
\end{example}

In a similar manner we may obtain the following interesting relation between the probability of a path and a star with the same number of vertices and edges. We omit the details,  as the argument is analogous to that in the previous example,  with the only difference being that a larger number of clique covers need to be considered.

\begin{example}\label{Ex:Star:Path}
	For $mp^3=O(1)$ and $m\to \infty$ we have
	\begin{equation}
		\pi(K_{1, 3})-\pi(P_4)=
		\left(1+O(m^{-\delta_4})\right)e^{-2mp^2+mp^3}\left(1-e^{-mp^3}\right)\left(1-e^{-mp^2+2mp^3}\right).
	\end{equation}
	In particular,
	\begin{equation}
		\pi(K_{1, 3})-\pi(P_4)=
		\begin{cases}
			\left(1+O(m^{-\delta_4}+mp^2)\right)m^2p^5&\text{ for }mp^2=o(1);\\
			\left(1+O(m^{-\delta_4}+mp^3)\right)mp^3e^{-2mp^2}(1-e^{-mp^2})&\text{ for }mp^3=o(1).\\
		\end{cases}
	\end{equation}		

\end{example}

\section*{Acknowledgements}	

Grzegorz Serafin's work was supported by the National Science Centre (NCN),  Poland,   under  research grant no. 2021/43/D/ST1/03244. The funder had no role in study design,  data collection and analysis,  decision to publish,  or preparation of the manuscript.


\begin{thebibliography}{10}
	
	\bibitem{RIGClustering2}
	Mindaugas Bloznelis.
	\newblock Degree and clustering coefficient in sparse random intersection
	graphs.
	\newblock {\em The Annals of Applied Probability}, 23(3):1254--1289, 2013.
	
	\bibitem{RIGsurvey1}
	Mindaugas Bloznelis, Erhard Godehardt, Jerzy Jaworski, Valentas Kurauskas, and
	Katarzyna Rybarczyk.
	\newblock Recent progress in complex network analysis: Models of random
	intersection graphs.
	\newblock In Berthold Lausen, Sabine Krolak-Schwerdt, and Matthias B{\"o}hmer,
	editors, {\em Data Science, Learning by Latent Structures, and Knowledge
		Discovery}, Studies in Classification, Data Analysis, and Knowledge
	Organization, pages 69--78. Springer Berlin Heidelberg, 2015.
	
	\bibitem{RIGsurvey2}
	Mindaugas Bloznelis, Erhard Godehardt, Jerzy Jaworski, Valentas Kurauskas, and
	Katarzyna Rybarczyk.
	\newblock Recent progress in complex network analysis: Properties of random
	intersection graphs.
	\newblock In Berthold Lausen, Sabine Krolak-Schwerdt, and Matthias B{\"o}hmer,
	editors, {\em Data Science, Learning by Latent Structures, and Knowledge
		Discovery}, Studies in Classification, Data Analysis, and Knowledge
	Organization, pages 79--88. Springer Berlin Heidelberg, 2015.
	
	\bibitem{WSNphase2}
	Mindaugas Bloznelis, Jerzy Jaworski, and Katarzyna Rybarczyk.
	\newblock Component evolution in a secure wireless sensor network.
	\newblock {\em Networks}, 53(1):19--26, 2009.
	
	\bibitem{BrennanBresterNagaraj2020Equivalence}
	Matthew Brennan, Guy Bresler, and Dheeraj Nagaraj.
	\newblock Phase transitions for detecting latent geometry in random graphs.
	\newblock {\em Probability Theory and Related Fields}, 178(3–4):1215--1289,
	September 2020.
	
	\bibitem{GpEpidemics}
	Tom Britton, Maria Deijfen, Andreas~N. Lager{\aa}s, and Mathias Lindholm.
	\newblock Epidemics on random graphs with tunable clustering.
	\newblock {\em Journal of Applied Probability}, 45(3):743--756, 2008.
	
	\bibitem{RIGTunableDegree}
	Maria Deijfen and Willemien Kets.
	\newblock Random intersection graphs with tunable degree distribution and
	clustering.
	\newblock {\em Probability in the Engineering and Information Sciences},
	23:661--674, 2009.
	
	\bibitem{Dong2023}
	Liang Dong and Zhishui Hu.
	\newblock The number of triangles in random intersection graphs.
	\newblock {\em Communications in Mathematics and Statistics}, 11(4):695--725,
	January 2023.
	
	\bibitem{FillScheinSinger2000Equivalence}
	James~Allen Fill, Edward~R. Scheinerman, and Karen~B. Singer-Cohen.
	\newblock Random intersection graphs when $m=\omega(n)$: An equivalence theorem
	relating the evolution of the {G}(n, m, p) and {G}(n, p) models.
	\newblock {\em Random Structures \& Algorithms}, 16:156--176, 2000.
	
	\bibitem{RIGGodehardt1}
	Erhard Godehardt and Jerzy Jaworski.
	\newblock Two models of random intersection graphs for classification.
	\newblock In O.~Opitz and M.~Schwaiger, editors, {\em Studies in Classifcation,
		Data Analysis and Knowledge Organization}, volume~22, pages 67--81. Springer,
	2003.
	
	\bibitem{KsiazkaJLR}
	Svante Janson, Tomasz \L~uczak, and Andrzej Rucinski.
	\newblock {\em Random graphs}.
	\newblock Wiley-Interscience Series in Discrete Mathematics and Optimization.
	Wiley-Interscience, New York, 2000.
	
	\bibitem{1999KarSchSin}
	Micha{\l} Karo\'{n}ski, Edward~R. Scheinerman, and Karen~B. Singer-Cohen.
	\newblock On random intersection graphs: The subgraph problem.
	\newblock {\em Combinatorics, Probability and Computing}, 8(1–2):131–159,
	1999.
	
	\bibitem{KimLeeNa2018Equivalence}
	Jeong~Han Kim, Sang~June Lee, and Joohan Na.
	\newblock On the total variation distance between the binomial random graph and
	the random intersection graph.
	\newblock {\em Random Structures \& Algorithms}, 52(4):662--679, January 2018.
	
	\bibitem{Rybarczyk2011Equivalence}
	Katarzyna Rybarczyk.
	\newblock Equivalence of the random intersection graph and {G}(n,p).
	\newblock {\em Random Structures \& Algorithms}, 38:205--234, 2011.
	
	\bibitem{2017RybSta}
	Katarzyna Rybarczyk and Dudley Stark.
	\newblock Poisson approximation of counts of induced subgraphs in random
	intersection graphs.
	\newblock {\em Discrete Mathematics}, 340(9):2183--2193, 2017.
	
	\bibitem{Spirakis2013}
	P.~Spirakis, S.~Nikoletseas, and Ch. Raptopoulos.
	\newblock A guided tour in random intersection graphs.
	\newblock In F.V. Fomin, R.~Freivalds, M.~Kwiatkowska, and D.~Peleg, editors,
	{\em Automata, Languages, and Programming. ICALP 2013.}, volume 7966 of {\em
		Lecture Notes in Computer Science}, pages 29--35, Springer, Berlin,
	Heidelberg, 2013.
	
\end{thebibliography}

\end{document}